\documentclass[11pt]{amsart}
\usepackage[active]{srcltx}
\usepackage[utf8]{inputenc}
\usepackage{amsmath,mathtools}
\usepackage{stmaryrd}
\usepackage{MnSymbol}
\usepackage{hyperref}
\usepackage{tikz}
\usetikzlibrary{arrows,automata}
\usepackage{enumitem}
\usepackage{mathrsfs}
\usepackage{pst-plot}
\usepackage{auto-pst-pdf}
\usepackage[all]{xy}
\usepackage{stackrel}
\usepackage{listings}
\usepackage{rotating}
\usepackage{graphicx}
\lstdefinelanguage{GAP}{%
  morekeywords={%
    Assert,Info,IsBound,QUIT,%
    TryNextMethod,Unbind,and,break,%
    continue,do,elif,%
    else,end,false,fi,for,%
    function,if,in,local,%
    mod,not,od,or,%
    quit,rec,repeat,return,%
    then,true,until,while%
  },%
  sensitive,%
  morecomment=[l]\#,%
  morestring=[b]",%
  morestring=[b]',%
}[keywords,comments,strings]
\usepackage[T1]{fontenc}
\usepackage{xcolor}
\usepackage{enumitem}
\newtheorem{thm}{Theorem}[section]

\newtheorem{deff}[thm]{Definition}

\newtheorem{lem}[thm]{Lemma}
\newtheorem{prop}[thm]{Proposition}
\newtheorem{example}[thm]{Example}

\newtheorem{con}[thm]{Conjecture}

\newcommand{\abs}[1]{\left\vert#1\right\vert}

\newcommand{\cc}[1]{\textcolor{red}{{#1}}}

\newcommand{\qq}[1]{\textcolor{magenta}{{#1}}}

\newcommand{\vv}[1]{\textcolor{cyan}{{#1}}}
\def\pv#1{\ensuremath{\mathsf{#1}}}
\newcommand{\R}{\mathrel{\mathscr R}} 
\newcommand{\eL}{\mathrel{\mathscr L}} 
\newcommand{\HH}{\mathrel{\mathscr H}}

\newcommand{\DDet}{\mathop{\mathrm{Det}}\nolimits}

\usepackage{caption}

\makeatletter
\newcommand{\mylabel}[2]{#2\def\@currentlabel{#2}\label{#1}}
\makeatother

\begin{document}
\title[The determinant of \(\lll\)-smooth semigroups]{The determinant of \(\lll\)-smooth semigroups}
\author{M.H. Shahzamanian}
\address{M.H. Shahzamanian\\ CMUP, Departamento de Matemática, Faculdade de Ciências, Universidade do Porto, Rua do Campo Alegre s/n, 4169--007 Porto (Portugal).}
\email{m.h.shahzamanian@gmail.com}

\begin{abstract}
This paper continues the investigation of non-zero determinants associated with finite semigroups containing a pair of non-commuting idempotents, as initiated in~\cite{Sha-Det2}. We focus on a class of semigroups, called \( \lll \)-smooth semigroups, that allow meaningful structural analysis despite the absence of \( \ll \)-transitivity. Within this framework, we develop a method for computing contracted semigroup determinants, building on and extending the approach carried over from~\cite{Sha-Det2}.
These computations are motivated by applications in coding theory, particularly by the potential extending the MacWilliams theorem for codes over semigroup algebras.
\end{abstract}
\maketitle


\section{Introduction}

In the 1880s, Dedekind introduced the concept of the group determinant for finite groups, which was later studied in depth with Frobenius. 
At the same time, Smith approached the concept from a different perspective, focusing on the determinant of a \(G \times G\) matrix where the entry at position \((g, h)\) is \(x_{gh}\), with \(G\) being a finite group and \(x_k\) as variables for each \(k \in G\) \cite{Smith}. 
This concept has since been extended to finite semigroups, with various researchers contributing to its development \cite{Lindstr, Wilf, Wood}. 
A notable application of the semigroup determinant arises in coding theory. The extension of the MacWilliams theorem from codes over finite fields to chain rings relies on the non-zero determinant property of certain semigroups. Specifically, the theory of linear codes over finite Frobenius rings demonstrates this extension property, where the semigroup determinant plays a key role \cite{Wood-Duality}.  


In \cite{Ste-Fac-det}, Steinberg examined the factorization of the semigroup determinant for commutative semigroups, proving that it is either zero or factors into linear polynomials.  
Building on this, \cite{Sha-Det} investigates the determinant of semigroups in the pseudovariety \(\pv{ECom}\), which consists of semigroups whose idempotents commute.  
As shown by Ash, \(\pv{ECom}\) is generated by finite inverse semigroups—a broader class than the semigroups with central idempotents considered in \cite{Ste-Fac-det}.  

In \cite{Sha-Det2}, the study was extended to semigroups beyond \(\pv{ECom}\), focusing on finite semigroups possessing pairs of non-commutative idempotents.  
The focus was on a class of semigroups that, while not in \(\pv{ECom}\), satisfy certain structural properties such as \(\ll\)-smoothness, \(\ll\)-transitivity, and being a singleton-rich monoid.  
In \cite{Sha-Det3}, a class of monoids referred to as Layered Catalan Monoids (\(LC_n\)), satisfying the conditions for \(\ll\)-smoothness, was introduced.  
These monoids are characterized by specific algebraic identities, including those defining Catalan monoids, with some exceptions.  
Furthermore, in \cite{Sha-Det3}, the determinant of \(LC_n\) was computed, proving that it is nonzero for \(1 \leq n \leq 7\) but vanishes for \(n \geq 8\).

The key to computing the determinant of the semigroups studied in \cite{Ste-Fac-det, Sha-Det, Sha-Det2} lies in equipping the semigroup \( S \) with a partially ordered set structure, along with a map \( Z \) defined based on this order.  
The map \( Z \) sends each element of \( S \) to the sum of all elements in \( S \) that are less than it with respect to the partial order.  
A crucial aspect of this method is that \( Z \) induces an isomorphism of \( \mathbb{C} \)-algebras, thereby establishing a connection between the determinants of the original semigroup algebra and its image under \( Z \).  
Since the determinant of the image algebra is typically much easier to compute, this isomorphism enables the calculation of the determinant of the original semigroup algebra via its image.

For the semigroups considered in~\cite{Ste-Fac-det, Sha-Det, Sha-Det2}, such a connection has been firmly established.  
In particular,~\cite{Sha-Det2} demonstrates the existence of this connection for \(\ll\)-transitive semigroups.  
However, for semigroups that are not \(\ll\)-transitive, this question remains open.  
In this paper, we explore a broader class of semigroups in which the connection still holds, despite the absence of \(\ll\)-transitivity.  
We refer to these as \emph{pseudo \(\ll\)-transitive semigroups}.  
Within this broader class, we focus on a subclass that we term \emph{\(\lll\)-smooth semigroups}, for which we compute the determinant.  
This subclass provides a small but significant family of non-\(\ll\)-transitive semigroups whose determinants are now better understood.

In \cite{Sha-Det}, it is shown that for a semigroup \( S \in \pv{ECom} \), the isomorphic image \( (S, \ast) \) of \( (S, \cdot) \) under the map \( Z \) does not contain elements \( s \) and \( t \) such that \( s^{\ast} \neq t^{+} \) and \( s \boldsymbol{*} t \neq 0 \).  
Although \cite{Sha-Det2} establishes that the mapping \( Z \) is an isomorphism of \( \mathbb{C} \)-algebras for \(\ll\)-transitive semigroups, it does not explore how this algebraic property influences or fails to prevent the existence of such elements in semigroups outside the pseudovariety \( \pv{ECom} \).
In this paper, we address this question for a broader class of semigroups beyond the \(\ll\)-transitive case specifically, those that are \emph{singleton-rich}.  
We demonstrate that for such semigroups, there do exist elements \( s \) and \( t \) such that \( s^{\ast} \neq t^{+} \) and \( s \boldsymbol{*} t \neq 0 \).  
This result reveals a fundamental distinction between semigroups within and beyond \( \pv{ECom} \), particularly in how their algebraic structure interacts with combinatorial properties.


\section{Preliminaries}

\subsection{Semigroups}

For standard notation and terminology regarding semigroups, we refer the reader to~\cite[Chap.~5]{Alm}, \cite[Chaps.~1--3]{Cli-Pre}, and~\cite[Appendix~A]{Rho-Ste}.

Let \( S \) be a finite semigroup, and let \( a, b \in S \). Green’s relations, originally introduced by Green~\cite{Gre}, are defined as follows:
\begin{itemize}
    \item \( a \mathrel{\mathcal{R}} b \) if and only if \( aS^1 = bS^1 \),
    \item \( a \mathrel{\mathcal{L}} b \) if and only if \( S^1a = S^1b \),
    \item \( a \mathrel{\mathcal{H}} b \) if and only if \( a \mathrel{\mathcal{R}} b \) and \( a \mathrel{\mathcal{L}} b \),
    \item \( a \mathrel{\mathcal{J}} b \) if and only if \( S^1aS^1 = S^1bS^1 \).
\end{itemize}
If \( S \) contains an identity element, then we set \( S^1 = S \). Otherwise, we define \( S^1 = S \cup \{1\} \), where \( 1 \) acts as an identity. 
The \( \mathcal{R}, \mathcal{L}, \mathcal{H}, \mathcal{J} \)-classes of an element \( a \in S \) are denoted by \( R_a, L_a, H_a \), and \( J_a \), respectively.

We also consider a refinement of the \( \mathcal{L} \)-relation, introduced by Fountain et al.~\cite{Fou-Gom-Gou}. We write \( a \mathrel{\widetilde{\mathcal{L}}} b \) if and only if \( a \) and \( b \) share the same set of idempotent right identities, that is,
\[
ae = a \quad \text{if and only if} \quad be = b.
\]
The dual relation \( \widetilde{\mathcal{R}} \) is defined analogously, and we set \( \widetilde{\mathcal{H}} = \widetilde{\mathcal{L}} \cap \widetilde{\mathcal{R}} \). The corresponding equivalence classes of an element \( s \in S \) are denoted by \( \widetilde{L}_s, \widetilde{R}_s \), and \( \widetilde{H}_s \), respectively. For additional background, see~\cite{Lawson-Sem-Cat}.

An element \( e \in S \) is called \emph{idempotent} if \( e^2 = e \). The set of all idempotents in \( S \) is denoted by \( E(S) \). 

For any element \( s \in S \), we define \( s^{\omega} \) as the eventual value of the sequence \( (s^{n!})_n \), which exists since \( S \) is finite.


\subsection{Identities}

For standard notation and terminology related to  identities, we refer the reader to~\cite{MR912694}.

Let \( X \) be a countably infinite set, called an \emph{alphabet}, and its elements are referred to as \emph{letters}.  
The set \( X^+ \) consists of all finite, non-empty words \( x_1 \cdots x_n \) with \( x_1, \ldots, x_n \in X \), and forms a semigroup under concatenation, known as the \emph{free semigroup} over \( X \).  
The corresponding \emph{free monoid} \( X^* = (X^+)^1 \) includes the empty word, which serves as the identity element.

Let \( t = x_1 \cdots x_n \in X^+ \), with each \( x_i \in X \).  
The set \( \{x_1, \ldots, x_n\} \) is called the \emph{content} of \( t \), denoted by \( c(t) \), and the number \( n \) is the \emph{length} of \( t \), denoted \( \abs{t} \).  
A letter \( x \in X \) is said to \emph{occur} in \( t \) if \( x \in c(t) \).  
We say that a word \( s \in X^+ \) \emph{occurs in} \( t \) if \( t = t_1 s t_2 \) for some \( t_1, t_2 \in X^* \).

Let \( u = u_1 \cdots u_m \in X^* \), with each \( u_i \in X \).  
We say that \( u \) is a \emph{subword} of \( t \in X^* \) if there exist words \( u_0', u_1', \ldots, u_m' \in X^* \) such that
\[
t = u_0' u_1 u_1' u_2 \cdots u_{m-1}' u_m u_m'.
\]
Given a subset \( Y \subseteq X \), we denote by \( t_Y \) the longest subword of \( t \) such that \( c(t_Y) \subseteq Y \).

An \emph{identity} is an equation of the form \( t_1 = t_2 \), where \( t_1, t_2 \in X^* \).  
A monoid \( M \) is said to \emph{satisfy} the identity \( t_1 = t_2 \), or that the identity \emph{holds in} \( M \), if for every homomorphism \( \phi \colon X^* \to M \), we have \( \phi(t_1) = \phi(t_2) \).


\subsection{Determinant of a semigroup}
For standard notation and terminology relating to finite dimensional algebras, the reader is referred to \cite{Assem-Ibrahim, Benson}.

A based algebra is a finite-dimensional complex algebra \( A \) equipped with a distinguished basis \( B \). We often refer to the pair \( (A, B) \) to emphasize both the algebra and its basis.  
The multiplication in \( A \) is determined by the structure constants with respect to \( B \), given by the equations  
\[
bb' = \sum\limits_{b''\in B} c_{b'',b,b'} b'',
\]
where \( b, b' \in B \) and \( c_{b'',b,b'} \in \mathbb{C} \).  

Let \( X_B = \{ x_b \mid b \in B \} \) be a set of variables in bijection with \( B \). These structure constants can be represented in a matrix called the Cayley table, which is a \( B \times B \) matrix with entries in the polynomial ring \( \mathbb{C}[X_B] \), defined by  
\[
C(A,B)_{b,b'} = \sum\limits_{b''\in B} c_{b'',b,b'} x_{b''}.
\]  
The determinant of this matrix, denoted by $\theta_{(A,B)}(X_B)$, is either identically zero or a homogeneous polynomial of degree $\abs{B}$.


Let \( S \) be a finite semigroup. The semigroup \( \mathbb{C} \)-algebra \( \mathbb{C}S \) consists of all formal sums  
\[
\sum\limits_{s\in S} \lambda_s s,
\]
where \( \lambda_s \in \mathbb{C} \) and \( s \in S \), with multiplication defined by the formula  
\[
\left( \sum\limits_{s\in S} \lambda_s s \right) \cdot \left( \sum\limits_{t\in S} \mu_t t \right) = \sum\limits_{u=st\in S} \lambda_s \mu_t u.
\]
Note that \( \mathbb{C}S \) is a finite-dimensional \( \mathbb{C} \)-algebra with basis \( S \).  

If we set \( A = \mathbb{C}S \) and \( B = S \), then the Cayley table \( C(S) = C(\mathbb{C}S, S) \) is the \( S \times S \) matrix over \( \mathbb{C}[X_S] \) with  
\[
C(S)_{s,s'} = x_{ss'},
\]
where \( X_S = \{ x_s \mid s \in S \} \) is a set of variables in bijection with \( S \).  
We denote the determinant \( \DDet C(\mathbb{C}S, S) \) by \( \theta_S(X_S) \) and refer to it as the (Dedekind-Frobenius) \emph{semigroup determinant} of \( S \).  
If the semigroup \( S \) is fixed, we often write \( X \) instead of \( X_S \).  

For more details on this topic, we refer the reader to \cite{Frobenius1903theorie}, \cite[Chapter~16]{Okn}, and \cite{Ste-Fac-det}.  


The \emph{contracted semigroup algebra} of a semigroup \( S \) with a zero element \( 0 \) over the complex numbers is defined as  
\[
\mathbb{C}_0S = \mathbb{C}S / \mathbb{C}0;
\]
note that \( \mathbb{C}0 \) is a one-dimensional two-sided ideal.  
This algebra can be thought of as having a basis consisting of the non-zero elements of \( S \) and having multiplication that extends that of \( S \), but with the zero of the semigroup being identified with the zero of the algebra.  

The \emph{contracted semigroup determinant} of \( S \), denoted by \( \widetilde{\theta}_S \), is the determinant of  
\[
\widetilde{C}(S) = C(\mathbb{C}_0S, S\setminus\{0\}),
\]
where  
\[
\widetilde{C}(S)_{s,t} =
\begin{cases}
    x_{st}, & \text{if } st \neq 0, \\
    0, & \text{otherwise}.
\end{cases}
\]
Let \( \widetilde{X} = X_{S \setminus \{0\}} \) if \( S \) is understood.  

According to Proposition 2.7 in~\cite{Ste-Fac-det} (the idea mentioned in~\cite{Wood}), there is a connection between the contracted semigroup determinant and the semigroup determinant of a semigroup \( S \) with a zero element.  
There is a \( \mathbb{C} \)-algebra isomorphism between the \( \mathbb{C} \)-algebra \( \mathbb{C}S \) and the product algebra \( \mathbb{C}_0S \times \mathbb{C}0 \), which sends \( s \in S \) to \( (s, 0) \).  
Put \( y_s = x_s - x_0 \) for \( s \neq 0 \) and let \( Y = \{ y_s \mid s \in S \setminus \{0\} \} \).  
Then  
\[
\theta_S(X) = x_0 \widetilde{\theta}_S(Y).
\]  
Therefore, \( \widetilde{\theta}_S(\widetilde{X}) \) can be obtained from \( \theta_S(X) / x_0 \) by replacing \( x_0 \) with 0.  


\section{Relation \(\lll\)}\label{llAndlll}

Necessary conditions for a semigroup $S$ to have a nonzero $\theta_S(X)$ are stated in \cite{Ste-Fac-det}.
If $\theta_S(X)$ is not equal to 0, then the semigroup algebra $\mathbb{C}S$ is a unital algebra, according to Theorem 2.1 in~\cite{Ste-Fac-det}.

We define the functions $\varphi^{\ast}$ and $\varphi^{+}$ from $S$ to the power set of $E(S)$ 
as follows:
\[\varphi^{\ast}(s)=\{e\in E(S)\mid se = s\}\ \text{and} \ \varphi^{+}(s)=\{e\in E(S)\mid es = s\}.\]
If $S$ is finite and $\mathbb{C}S$ is a unital algebra, then the subsets $\varphi^{\ast}(s)$ and $\varphi^{+}(s)$ are nonempty, for every $s\in S$ (\cite[Lemma 3.1]{Sha-Det}).

In~\cite{Sha-Det}, the determinant of a semigroup within the pseudovariety \pv{ECom} is explored.
In~\cite{Sha-Det2}, the investigation into the determinant of a semigroup outside the pseudovariety \pv{ECom} began. In this paper, we continue this investigation and consider additional cases.
Throughout the paper, we consider a finite semigroup 
$S$ with the assumption that the semigroup algebra $\mathbb{C}S$ is a unital algebra.
In this section, we recall some definitions and results from Sections 3 and 4 of~\cite{Sha-Det2}.

Let $s\in S$. Since $\mathbb{C}S$ is a unital algebra,
the subsets $\varphi^{\ast}(s)$ and $\varphi^{+}(s)$ are nonempty.
We denote the kernel of $\langle \varphi^{\ast}(s) \rangle$ and $\langle \varphi^{+}(s) \rangle$ by $s^{\ast\ast}$ and $s^{++}$, respectively.
We have $s^{\ast\ast} = t^{\ast\ast}$ if and only if $\varphi^{\ast}(s)=\varphi^{\ast}(t)$. 
Additionally, we have $s^{++} = t^{++}$ if and only if $\varphi^{+}(s)=\varphi^{+}(t)$. 
Then, the equivalence relations $\widetilde{\eL}$, $\widetilde{\R}$
and $\widetilde{\HH}$ can be described as follows:
\begin{enumerate}
\item $s \widetilde{\eL} t$ if $s^{\ast\ast} = t^{\ast\ast}$;
\item $s \widetilde{\R} t$ if $s^{++} = t^{++}$;
\item $s \widetilde{\HH} t$ if $s^{\ast\ast} = t^{\ast\ast}$ and $s^{++} = t^{++}$.
\end{enumerate}

It is clear that if $e$ is an idempotent then $e \in e^{\ast\ast}, e^{++}$.

Let $s$ and $t$ be elements of $S$. We define $s \ll t$ if \[s = s^{++}ts^{\ast\ast}.\]

We say that a semigroup $S$ is singleton-rich
if the cardinality of subsets $s^{\ast\ast}$ and $s^{++}$ is equal to one, for every $s$ in $S$.
In this case, we denote the single element of the subsets $s^{\ast\ast}$ and $s^{++}$ by $s^{\ast}$ and $s^{+}$, respectively.
Note that, it is easy to verify that $s^{\ast}$ and $s^{+}$ are idempotent. 
In this paper, as in \cite{Sha-Det2}, we assume that all semigroups are singleton-rich.
It is easy to verify that if there exist idempotents $e$ and $f$ such that $ef = e$ and $fe = f$ (or the dual conditions $ef = f$ and $fe = e$), then $S$ is not singleton-rich.

By Lemma~3.1 and Proposition~3.2 in~\cite{Sha-Det2}, we have the following lemma.

\begin{lem}\label{eslls}
The following statements hold:
\begin{enumerate}
\item
Let $s\in S$ and $e, f\in E(S)$.
We have $es, se, esf \ll s$.
\item Let $e,f\in E(S)$. If $ef=e$ or $fe=e$, then we have $e\leq f$.
\item there do not exist pairwise distinct elements $s_1,\ldots,s_n$ with $1 < n$ such that
\[s_1 \ll s_2 \ll \cdots \ll s_n \ll s_1.\]
\end{enumerate}
\end{lem}


%
We examine cases where \( \ll \) may fail to be transitive under certain specified conditions. 
These conditions naturally lead us to define and study a class of semigroups, which we refer to as \( \lll \)-smooth semigroups. 
We then proceed to compute their determinant and explore its properties.


The relation $\ll$ is reflexive and antisymmetric (\cite[Proposition~3.2.(2)]{Sha-Det2}). 
However the relation $\ll$ may not be transitive (\cite[Example~6.1]{Sha-Det2}). 
Therefore, we define $\lll$ as the smallest partially ordered set containing $\ll$.


Let $t', t \in S$ with $t' \lll t$.  
We say that the pair $(t,t')$ is \emph{$n$-essential} if there exist elements $v_1, v_2, \dots, v_n$ such that  
\[
t' \ll v_1 \ll v_2 \ll \cdots \ll v_n \ll t,
\]
but no sequence $w_1, w_2, \dots, w_m$ with $m < n$ satisfies  
\[
t' \ll w_1 \ll w_2 \ll \cdots \ll w_m \ll t.
\]
In particular, if \( (t, t') \) is \emph{$1$-essential}, then there exists an element \( v_1 \) such that  
\(
t' \ll v_1 \ll t,
\)
but \( t \not\ll t' \).


We say that \(S\) is pseudo $\ll$-transitive, if 
for every elements $u, s, t\in S$ with $u \lll st$ and $u \not\ll st$ and $(u,st)$ is $n$-essential,
there exist elements $v_1,\ldots,v_n$ such that
\[
u\ll v_1 \ll \cdots \ll v_n \ll st
\]
and that satisfy the following conditions~\mylabel{pseudo}{($\#$)}:
\begin{enumerate}
\item \(u^{+}=v_1^{+}=\cdots =v_n^{+}=\varphi^{(u^{+} s ,t u^{\ast})} (= \varphi)\),
\item \( \varphi^{(u^{+} s ,t v_n^{\ast}\cdots v_1^{\ast}u^{\ast})} =  \varphi\),
\item \(v_n \ll t\).
\end{enumerate}

When \( u \lll t \) with \( u \not\ll t \) for some elements \( u \) and \( t \) in \( S \), the condition above ensures that there exist elements \( v_1, \ldots, v_n \) such that  
\[ 
u \ll v_1 \ll \cdots \ll v_n \ll t 
\]  
with \( u^{+} = v_1^{+} = \cdots = v_n^{+} \). 
We use $t^{+}$ in place of $s$. 

We used a C$\#$ program to verify the following conjecture.

\begin{con}\label{pseudolltransitive}
If $\abs{S} \leq 8$, then $S$ is pseudo $\ll$-transitive. 
\end{con}


We recall some definitions from \cite{Sha-Det2} and adapt them for $\ll$-transitive semigroups.

We define the function $\varphi$ 
as follows:
\[\varphi(s,t)=(s^{\ast}t)^{+},\]
for every $s,t\in S$.

Now, for $s$ and $t$ in $S$, based on the function $\varphi$, 
we define the sequences $s_i$ and $t_i$ 
as follows:
\begin{center}
\begin{tabular}{ll}
$s_0=s$, & $t_0=t;$\\
\text{and} &\\
$s_{i+1}=s\varphi(s_i,t_i)$, & $t_{i+1} = \varphi(s_i,t_i)t$;
\end{tabular}
\end{center}

By~\cite[Lemmas~4.4]{Sha-Det2}, the sequences $\varphi(s_i,t_i)$ converges to an equal element. 
We denote the converge of these sequences by $\varphi^{(s,t)}$. 

Let $Z\colon \mathbb{C}S \rightarrow \mathbb{C}S$ be a map given by $Z(s) =\sum\limits_{s'\lll s}s'$ on $s \in S$ with a linear extension.

We define the multiplication
$\sharp \colon \mathbb{C}S \times \mathbb{C}S \rightarrow \mathbb{C}S$ 
as follows:\\
\begin{equation*}
s\sharp t = 
\begin{cases}
  st,& \text{if}\ s^{+}=(st)^{+}, t^{\ast}=(st)^{\ast}\ \text{and}\ s^{\ast}=t^{+};\\ 
  0,& \text{otherwise},
\end{cases}
\end{equation*}
for every $s,t\in S$. 
Also, we define the multiplication
$\sharp\limits^{e} \colon \mathbb{C}S \times \mathbb{C}S \rightarrow \mathbb{C}S$, for some $e\in E(S)$,
as follows:\\
\begin{equation*}
s\sharp^{e} t = 
\begin{cases}
  st & \text{if}\ s^{+}=(st)^{+}, t^{\ast}=(st)^{\ast}\ \text{and}\ s^{\ast}=t^{+}=e;\\
  0& \text{otherwise}.
\end{cases}
\end{equation*}


We define $(s',t')\ll (s,t)$ if the following conditions hold:
\begin{enumerate}
\item \(s' \ll s\),
\item \(t' \lll t\),
\item $s'
\stackrel{\varphi}{\sharp}
t'
\neq 
0$ where \(\varphi = \varphi^{
({s'}^{+} s,t{t'}^{\ast})
}\),
\item If \( t' \not\ll t \), then \(s't'\not\ll st\), \(s't'=t'\) and there exist elements \(t_1,\ldots,t_n\), referred to as \mylabel{pairll}{(\(\ast\))}, such that
\begin{enumerate}
\item \(t'\ll t_1 \ll t_2 \ll \cdots \ll t_n \ll t\), 
\item \({t'}^{+}=t_1^{+}=t_2^{+}=\cdots =t_n^{+} = \varphi\),
\item \(\varphi^{
({s'}^{+} s,tt_n^{\ast}\cdots t_2^{\ast}t_1^{\ast}{t'}^{\ast})
} = 
\varphi
\).
\end{enumerate}
\end{enumerate}

Now, we define the following multiplication on $\mathbb{C}S \times \mathbb{C}S$
\begin{equation*}
Z(s)\boldsymbol{*}Z(t) = 
\sum\limits_{(s',t')\ll (s,t)} s't'.
\end{equation*}

\begin{prop}\label{s*tHom}
Suppose that \(S\) is pseudo $\ll$-transitive.
We have $Z(s)\boldsymbol{*}Z(t)= Z(st)$, for all $s,t\in S$.
\end{prop}

\begin{proof}
First, we prove that \( s't' \lll st \) if \( (s',t') \ll (s,t) \).

If \( t' \ll t \), then following the same reasoning as in~\cite[Proposition 5.2]{Sha-Det2}, we conclude that \( s't' \ll st \).

Now, suppose that \( t' \lll t \) with \( t' \not\ll t \).  
Then, there exists elements \( t_1, \ldots, t_n \) such that Condition~\ref{pairll} holds.  
Thus, we have  
\[
s' = {s'}^{+} s \varphi,
\quad \text{and} \quad
t' = {t'}^{+} t_1^{+} \cdots t_n^{+} t t_n^{\ast} \cdots t_1^{\ast} {t'}^{\ast} = \varphi t t_n^{\ast} \cdots t_1^{\ast} {t'}^{\ast}.
\] 

Hence,  
\[
s't' 
= {s'}^{+} s \varphi t t_n^{\ast} \cdots t_1^{\ast} {t'}^{\ast}
= {s'}^{+} s \varphi^{({s'}^{+} s,t t_n^{\ast} \cdots t_1^{\ast} {t'}^{\ast})} t t_n^{\ast} \cdots t_1^{\ast} {t'}^{\ast}
= {s'}^{+} s t t_n^{\ast} \cdots t_1^{\ast} {t'}^{\ast}.
\]

By Lemma~\ref{eslls}.(1), we conclude that \( s't' \lll st \).


Secondly, we need to show that if $u\lll st$, for some $u$ in $S$, there exists a unique pair $(s^{\circ},t^{\circ}) \ll (s,t)$ 
with $s^{\circ} t^{\circ} = u$.
We encounter two possibilities: either $u\ll st$ or $u\not\ll st$.


If $u\ll st$, a pair meeting this condition exists by following the same logic as the proof in~\cite[Proposition 5.2]{Sha-Det2}.

%


Now, suppose that \( u \not\ll st \). Since \( S \) is pseudo \(\ll\)-transitive,  
there exists elements \( v_1, \ldots, v_n \) such that  
\[
u \ll v_1 \ll \cdots \ll v_n \ll st,
\]
satisfying Conditions~\ref{pseudo}.

Thus, we have  
\[
u 
= 
u^{+} v_1^{+} \cdots v_n^{+} s t v_n^{\ast} \cdots v_1^{\ast} u^{\ast} = u^{+} s t v_n^{\ast} \cdots v_1^{\ast} u^{\ast}.
\]

Let 
\begin{align*}
s'     &= u^{+} s \varphi,\\
t'     &= \varphi t v_n^{\ast}\cdots v_1^{\ast}u^{\ast}, 
t_1= \varphi t v_n^{\ast}\cdots v_1^{\ast},
t_2= \varphi t v_n^{\ast}\cdots v_2^{\ast},
\ldots,
t_n= \varphi t  v_n^{\ast}
\end{align*}
where \(\varphi = \varphi^{(u^{+} s ,t u^{\ast})}\).

By Lemma~\ref{eslls}.(1), we deduce that  
\[
s' \ll s, \quad \text{and} \quad t' \ll t_1 \ll \cdots \ll t_n \ll t.
\]

Furthermore, we obtain  
\begin{align*}
s't' 
&= u^{+} s \varphi t v_n^{\ast} \cdots v_1^{\ast} u^{\ast} 
= u^{+} s \varphi^{(u^{+} s ,t v_n^{\ast} \cdots v_1^{\ast} u^{\ast})} t v_n^{\ast} \cdots v_1^{\ast} u^{\ast} \\
&= u^{+} s t v_n^{\ast} \cdots v_1^{\ast} u^{\ast} = u.
\end{align*} 
  
Additionally, by~\cite[Lemmas~4.7]{Sha-Det2}, we have  
\(
{s'}^{\ast} = {t'}^{+} = \varphi.
\)
Since  
\[
u = u^{+} s t v_n^{\ast} \cdots v_1^{\ast} u^{\ast} = u^{+} s \varphi t v_n^{\ast} \cdots v_1^{\ast} u^{\ast},
\]
it follows that \({s'}^{+} = u^{+}\) and \({t'}^{\ast} = u^{\ast}\), which implies  
\(
s' \stackrel{\varphi}{\sharp} t' \neq 0.
\)

Since \( v_n \ll t \), we have \( t_n = v_n \). Similarly, for all other indices \( i \), we get \( t_i = v_i \) and \( t' = u \).  
Thus, \(s't'=t'\) and we obtain  
\[
{t'}^{+} = t_1^{+} = t_2^{+} = \cdots = t_n^{+} = \varphi,
\]
and  
\[
\varphi^{
({s'}^{+} s, t t_n^{\ast} \cdots t_2^{\ast} t_1^{\ast} {t'}^{\ast})
} = 
\varphi.
\]

If \( t' \ll t \), then \( u=s't' \ll st \), a contradiction.  
Otherwise, the sequence \[ t' \ll t_1 \ll \cdots \ll t_n \ll t \] satisfies the other conditions,  
which implies that \( (s', t') \ll (s,t) \).
  

There is then our desired pair. 
Now, we prove the uniqueness of this existence by contradiction. 
In doing so, we consider distinct pairs \[(s'_1,t'_1), (s'_2,t'_2)\ll (s,t)\] in which \(s'_1, s'_2 \ll s\), \(t'_1, t'_2 \lll t\) and \(s'_1t'_1 = s'_2t'_2 = u\) with the following conditions, each leading to a contradiction:

(\(t'_1, t'_2 \ll t\)):\\
Following the same logic as the proof in~\cite[Proposition 5.2]{Sha-Det2}, we have $s'_1=s'_2$ and $t'_1=t'_2$. 

(\(t'_1 \ll t\), \(t'_2 \not\ll t\)):\\
Since \( (s'_1, t'_1) \ll (s,t) \) and \( t'_1 \ll t \), it follows that \( s'_1 t'_1 \ll st \) and, consequently, \( u \ll st \).  
Moreover, since \( (s'_2, t'_2) \ll (s,t) \) and \(t'_2 \not\ll t\), we obtain \( t'_2 = s'_2 t'_2 = u\), a contradiction with \( s'_2 t'_2 \not\ll st \).

(\(t'_1, t'_2 \not\ll t\)):\\
First, we have $t'_1=s'_1t'_1 =u= s'_2t'_2= t'_2$.

Since \(s'_1t'_1 = s'_2t'_2 = u\) and 
\(
s'_1\stackrel{\varphi^{({s'}_1^{+}s,t{t'}_1^{\ast})}}{\sharp} t'_1,
s'_2\stackrel{\varphi^{({s'}_2^{+}s,t{t'}_2^{\ast})}}{\sharp} t'_2
\neq 0\),
we have 
${s'}_1^{+}={s'}_2^{+}=u^{+}$. 
Moreover, we have
\({s'}_1^{\ast}= {s'}_2^{\ast}=  {t'}_1^{+} (=u^{+}).\)
Therefore, we have ${s'}_1=u^{+}su^{+}={s'}_2$.

The result follows.
\end{proof}

By applying Möbius inversion, the map \( Z \) is bijective. Thus, by Proposition~\ref{s*tHom}, we obtain the following theorem.

\begin{thm}\label{Zisom}
Let \(S\) be a pseudo $\ll$-transitive semigroup.
The mapping $Z$ is an isomorphism of $\mathbb{C}$-algebras. 
\end{thm}

We denote the isomorphic image of \( (S, \cdot) \) under \( Z \) by \( (S, \ast) \).

Let $s,t\in S$.
By applying M\"{o}bius inversion, we have  \[
s= \displaystyle\sum\limits_{s'\lll s}\mu_S(s', s)Z(s')\
\text{and}\ 
t= \displaystyle\sum\limits_{t'\lll t}\mu_S(t', t)Z(t').\]
Then, we get that
\begin{align*}
s\boldsymbol{*}t
&=
\displaystyle\sum\limits_{\substack{s'\lll s,\\ t'\lll t}}\mu_S(s', s)\mu_S(t', t)Z(s')\boldsymbol{*}Z(t')\\
&=
\displaystyle\sum\limits_{\substack{s'\lll s,\\ t'\lll t}}\mu_S(s', s)\mu_S(t', t)\sum\limits_{(s'',t'')\ll (s',t')} s''t''\\
&=
\displaystyle\sum\limits_{\substack{s'\lll s,\\ t'\lll t}}\mu_S(s', s)\mu_S(t', t)\sum\limits_{\substack{s''\lll s', t''\lll t',\\ (s'',t'')\ll (s',t')}} s''t''\\
&=
\sum\limits_{\substack{s''\lll s,\\ t''\lll t}}\big[\sum\limits_{\substack{s''\lll s'\lll s,\\ t''\lll t'\lll t,\\ 
(s'',t'')\ll (s',t')}}
\mu_S(s', s)\mu_S(t', t)
\big]
s''t''\\
&=
\sum\limits_{\substack{s''\lll s,\\ t''\lll t}}
\big[
\sum\limits_{s''\lll s'\lll s}
\big(
\sum\limits_{\substack{t''\lll t'\lll t,\\ (s'',t'')\ll (s',t')}}
\mu_S(t', t)
\big)
\mu_S(s', s)
\big]
s''t''
.
\end{align*}
We define the function \(\xi\colon S^{(s)}\times S^{(t)}\rightarrow \mathbb{C}\), 
where \(S^{(s)}=\{s'' \in S \mid s'' \lll s\}\), for every \(s \in S\), as follows:
\begin{align*}
\xi(s'',t'') 
&=
\sum\limits_{s''\lll s'\lll s}
\big(
\sum\limits_{\substack{t''\lll t'\lll t,\\ (s'',t'')\ll (s',t')}}
\mu_S(t', t)
\big)
\mu_S(s', s),
\end{align*}
for every \(s'' \in S^{(s)}\) and \(t'' \in S^{(t)}\).

A natural question that arises is whether a semigroup \( S \) with  
\( S \not\in \pv{ECom} \), \( S \) being singleton-rich, and the mapping \( Z \) being an isomorphism of \( \mathbb{C} \)-algebras  
contains elements \( s \) and \( t \) such that \( s^{\ast} \neq t^{+} \) and \( s \boldsymbol{*} t \neq 0 \),  
in contrast to semigroups in \( \pv{ECom} \), which do not contain such elements.  
This question was not mentioned or answered.

In the remainder of this section, we address this question.

Let
\begin{align*}
E_S=\{(e,f) \mid 
&\ e,f\in E(S), ef \neq fe\ \text{and there are no idempotents}\ e'\ \text{and}\ f'\\ 
&\ \text{such that}\ e'\leq e\ \text{and}\ f'\leq f, (e',f')\neq (e,f)\ \text{and}\ e'f'\neq f'e'\}.
\end{align*}

If $S$ is finite and $S \not\in \pv{ECom}$, then the subset $E_S$ is clearly non-empty.

\begin{lem}\label{splusneqstrastneqLem1}
Let $S \not\in \pv{ECom}$ where $S$ is singleton-rich. 
If $(e,f) \in E_S$, then one of the following conditions hold:
\begin{enumerate}
\item we have $(ef)^{+} = e$ and $(ef)^{\ast} = f$;
\item we have $(fe)^{+} = f$ and $(fe)^{\ast} = e$.
\end{enumerate}
Moreover, if $(ef)^{+} \neq e$ or $(ef)^{\ast} \neq f$, then $ef$ is an idempotent.
\end{lem}

\begin{proof}
Suppose that $(ef)^{+} \neq e$ and $(fe)^{+} \neq f$.

Since $(ef)^{+} \neq e$, $(ef)^{+} \leq e$ and $(e,f) \in E_S$, we have $(ef)^{+}f = f(ef)^{+}$. Then, we get that
\[ef = (ef)^{+}ef = (ef)^{+}f = (ef)^{+} f f = f (ef)^{+}f = f(ef)^{+}ef = fef.\]
Hence, $ef$ is an idempotent. It follows that $(ef)^{+} = ef$ and, thus, $ef \leq e$ with $ef \neq e$. 

Similarly, as $(fe)^{+} \neq f$, we have $fe=efe$. It follows that $(fe)^{+} = fe$ and, thus, $fe \leq f$ with $fe \neq f$.

Then, we get that \[feef = fef = ef \neq fe = efe = effe,\]
a contradiction with our assumption.

Therefore, we have $(ef)^{+} = e$ or $(fe)^{+} = f$.
Additionally, we can conclude that $(ef)^{\ast} = f$ or $(fe)^{\ast} = e$.
The result follows if either $(ef)^{+} = e$ and $(ef)^{\ast} = f$, or $(fe)^{+} = f$ and $(fe)^{\ast} = e$.

Suppose that $(ef)^{+} = e$, $(ef)^{\ast} \neq f$, $(fe)^{+} \neq f$ and $(fe)^{\ast} = e$.
As $(ef)^{\ast} \neq f$ and $(fe)^{+} \neq f$, then by a similar argument as above, we have $ef= efe = fe$, a contradiction.
%
Also, we have a contradiction, if $(ef)^{+} \neq e$, $(ef)^{\ast} = f$, $(fe)^{+} = f$ and $(fe)^{\ast} \neq e$.

The result follows.
\end{proof}

Note that in Lemma~\ref{splusneqstrastneqLem1}, there may exist idempotents $e$ and $f$ that satisfy the first condition but do not satisfy the second, or vice versa (see the following Example).

\begin{example}\label{Ex01}
Let \( S \) be a semigroup with the following Cayley table, where the dot represents the zero element of \( S \) in the table.
\begin{center}
\begin{tabular}{ c | c c c c } 
\textbf{\cc{$S$}} & $\cc{y}$ & $\cc{z}$ & $\cc{u}$ & $\cc{t}$\\ \hline 
$\cc{y}$  & . & . & . & $y$ \\ 
$\cc{z}$  & . & . & $z$  & .\\ 
$\cc{u}$  & $y$  & . & $u$  & .\\ 
$\cc{t}$  & . & $z$  & $z$  & $t$ \\ 
\end{tabular}
\end{center}
We have \( {z}^{+} = t \) and \( {z}^{\ast} = u \), but \( ut = 0 \).
\end{example}

\begin{lem}\label{splusneqstrastneqLem3}
Let \( S \) be a finite semigroup such that \( S \notin \pv{ECom} \) and \( S \) is singleton-rich.  
Then there exist distinct idempotents \( e \) and \( f \) in \( S \) such that \( ef \neq fe \),  
\( (ef)^{+} = e \), \( (ef)^{\ast} = f \), and the following holds:

\begin{itemize}
    \item there are no elements \( s_1, \ldots, s_n \) in \( S \), each distinct from both \( e \) and \( ef \), such that  
    \[
    ef \ll s_1 \ll \cdots \ll s_n \ll e,
    \]
    \item and no elements \( t_1, \ldots, t_m \) in \( S \), each distinct from both \( f \) and \( ef \), such that  
    \[
    ef \ll t_1 \ll \cdots \ll t_m \ll f.
    \]
\end{itemize}
\end{lem}

\begin{proof}
Since $S$ is finite, the subset $E_S$ is finite as well. 
Thus, there exists an idempotent \( e \) such that \( (e, f) \in E_S \) for some idempotent \( f \), and the following condition, described in ~\ref{itm:star}, holds:
\begin{enumerate}
\item[\mylabel{itm:star}{($\star$)}] if there are idempotents $e'$ and $g$ with $e' \leq e$ and $(e', g) \in E_S$, then $e' = e$.
\end{enumerate}

First, we prove that
\begin{enumerate}
\item[\mylabel{itm:starstar}{($\star\star$)}]
if $eh \neq he$ with $(eh)^{+} = e$ and $(eh)^{\ast} = h$ for some idempotent $h$,
and there exist elements $s_1, \ldots, s_n$ distinct from $e$ and $eh$ such that \[eh \ll s_1 \ll \cdots \ll s_n \ll e,\] 
then there exists an integer $1\leq i\leq n$ such that $es_i^{\ast} \neq s_i^{\ast}e$, $(es_i^{\ast})^{\ast} = s_i^{\ast}$, $(es_i^{\ast})^{+} = e$, $h \neq s_i^{\ast}$, and $eh \lll es_i^{\ast} \ll e$. 
\end{enumerate}


As $eh \ll s_1 \ll \cdots \ll s_n \ll e$, we have $s_n = s_n^{+}es_n^{\ast}$ and $eh = es_1^{+}\cdots s_n^{+}es_n^{\ast}\cdots s_1^{\ast}h$.

We prove by contradiction that there exists an integer $1\leq i\leq n$ such that $es_i^{\ast} \neq s_i^{\ast}e$.
Then, suppose that $es_i^{\ast} = s_i^{\ast}e$, for all $1\leq i\leq n$.
Hence, we have \[eh = es_1^{+}\cdots s_n^{+}es_n^{\ast}\cdots s_1^{\ast}h=es_1^{+}\cdots s_n^{+}s_n^{\ast}\cdots s_1^{\ast}eh.\] 
It follows that $eh = (es_1^{+}\cdots s_n^{+}s_n^{\ast}\cdots s_1^{\ast})^{\omega}eh$ and, thus, $e \leq (es_1^{+}\cdots s_n^{+}s_n^{\ast}\cdots s_1^{\ast})^{\omega}$.
As, $e(es_1^{+}\cdots s_n^{+}s_n^{\ast}\cdots s_1^{\ast})^{\omega}=(es_1^{+}\cdots s_n^{+}s_n^{\ast}\cdots s_1^{\ast})^{\omega}$, we get that $(es_1^{+}\cdots s_n^{+}s_n^{\ast}\cdots s_1^{\ast})^{\omega} = e$.
It follows that $e\leq s_1^{\ast}$. Hence, we have 
\[eh = es_1^{+}\cdots s_n^{+}s_n^{\ast}\cdots s_2^{\ast}eh.\] 
Again, we have $eh = (es_1^{+}\cdots s_n^{+}s_n^{\ast}\cdots s_2^{\ast})^{\omega}eh$ and, thus, $e = (es_1^{+}\cdots s_n^{+}s_n^{\ast}\cdots s_2^{\ast})^{\omega}$. Then, we get that $e\leq s_2^{\ast}$. Similarly, we conclude that $e\leq s_3^{\ast}, \ldots, s_n^{\ast},s_n^{+}$.

Thus, we have $s_n = e$, a contradiction.
Therefore, there exist an integer $i$ such that $es_i^{\ast} \neq s_i^{\ast}e$ and if $i<j\leq n$ then we have $es_j^{\ast} = s_j^{\ast}e$.

Now, we prove that $(es_i^{\ast})^{\ast} = s_i^{\ast}$.
As \[s_i = s_i^{+}\cdots s_n^{+}es_n^{\ast}\cdots s_i^{\ast}=s_i^{+}\cdots s_n^{+}s_n^{\ast}\cdots s_{i-1}^{\ast}es_i^{\ast},\]
we have $s_i^{\ast} \leq (es_i^{\ast})^{\ast}$. Hence, we have $(es_i^{\ast})^{\ast} = s_i^{\ast}$.

By contradiction, we prove that $(es_i^{\ast})^{+} = e$.
Suppose that $(es_i^{\ast})^{+} = e'$ with $e' \neq e$.
If $e's_i^{\ast} \neq s_i^{\ast}e'$, then there exist idempotents $e'' \leq e'$ and $s' \leq s_i^{\ast}$ such that $(e'', s') \in E_S$, which contradicts Condition~\ref{itm:star}.
Then, we have $e's_i^{\ast} = s_i^{\ast}e'$.
It follows that \[es_i^{\ast} = (es_i^{\ast})^{+}es_i^{\ast} = e'es_i^{\ast} = e's_i^{\ast} = s_i^{\ast}e'.\] 
Hence, $(es_i^{\ast})^{+} (=e') \leq s_i^{\ast}$ and, thus, $es_i^{\ast} = e'$.
Now, since $(es_i^{\ast})^{\ast} = s_i^{\ast}$, we have $e' = s_i^{\ast}$, which contradicts the assumption that $es_i^{\ast} \neq s_i^{\ast}e$.

If $h = s_i^{\ast}$, then we have 
\[eh 
= 
es_1^{+}\cdots s_n^{+}es_n^{\ast}\cdots s_{i+1}^{\ast}s_i^{\ast}s_{i-1}^{\ast}\cdots s_1^{\ast}h
=
es_1^{+}\cdots s_n^{+}s_n^{\ast}\cdots s_{i+1}^{\ast}ehs_{i-1}^{\ast}\cdots s_1^{\ast}h
,\] 
as we have $es_j^{\ast} = s_j^{\ast}e$, for $i<j\leq n$.
Then, we get that
\[eh 
= 
(es_1^{+}\cdots s_n^{+}s_n^{\ast}\cdots s_{i+1}^{\ast})^{\omega}eh(s_{i-1}^{\ast}\cdots s_1^{\ast}h)^{\omega}.\] 

It follows that $(es_1^{+}\cdots s_n^{+}s_n^{\ast}\cdots s_{i+1}^{\ast})^{\omega}=e$.
Hence, we have $e\leq s_{i+1}^{\ast}$. 
Similarly, we have \[e\leq s_{i+2}^{\ast},s_{i+3}^{\ast},\ldots, s_{n}^{\ast},s_n^{+}.\]

If $i\neq n$, then we have $s_n=s_n^{+}es_{n}^{\ast}=e$, a contradiction. Otherwise, we have $s_n=s_n^{+}es_{n}^{\ast}=eh$, again contradiction.

As $(es_i^{\ast})^{\ast} = s_i^{\ast}$, $(es_i^{\ast})^{+} = e$, we have $es_i^{\ast} \ll e$.
Also, as $es_j^{\ast} = s_j^{\ast}e$, for $i <j\leq n$, we have  
\[eh 
= 
es_1^{+}\cdots s_n^{+}es_n^{\ast}\cdots s_{i+1}^{\ast}s_i^{\ast}s_{i-1}^{\ast}\cdots s_1^{\ast}h
=
es_1^{+}\cdots s_n^{+}s_n^{\ast}\cdots s_{i+1}^{\ast}es_i^{\ast}s_{i-1}^{\ast}\cdots s_1^{\ast}h
.\] 
Then, by Lemma~\ref{eslls}.(1), we have $eh\lll es_i^{\ast}$.

Hence, \ref{itm:starstar} is established.


Since $(e, f) \in E_S$, by Lemma~\ref{splusneqstrastneqLem1}, by symmetry, we may assume that $(e, f)$ satisfies its first condition 
(where the second condition requires the dual version of \ref{itm:starstar}),
namely that $(ef)^{+} = e$ and $(ef)^{\ast} = f$.
Suppose that there exist elements \( s_1,\ldots,s_n \) distinct from \( e \) and \( ef \) such that 
\[ ef \ll s_1 \ll \cdots \ll s_n \ll e,\] Then by~\ref{itm:starstar}, there exists an integer $1\leq i\leq n$ such that \( es_i^{\ast} \neq s_i^{\ast}e \), \( (es_i^{\ast})^{\ast} = s_i^{\ast} \), \( (es_i^{\ast})^{+} = e \), \(f\neq s_i^{\ast}\), and $ef\lll es_i^{\ast}\ll e$. 
As \(f\neq s_i^{\ast}\), \( (es_i^{\ast})^{\ast} = s_i^{\ast} \) and \( (ef)^{\ast} = f \), we have \(ef \neq es_i^{\ast}\).

We may proceed by substituting \( f \) with \( s^{\ast} \) and continue the substitution. Since \( S \) is finite, by~\ref{itm:starstar}, exactly one of the following holds:
\begin{enumerate}
\item[\mylabel{itm:sharp1}{($\sharp1$)}] There exists an idempotent \( h \) such that \( eh \neq he \), with \( (eh)^{+} = e \) and \( (eh)^{\ast} = h \), and there do not exist element \( s_1,\ldots,s_m \) distinct from \( e \) and \( eh \) such that \( eh \ll s_1\ll\cdots\ll s_m \ll e \).

\item[\mylabel{itm:sharp2}{($\sharp2$)}] There exist pairwise distinct idempotents \( h_1, \ldots, h_n \) such that \( eh_i \neq h_i e \), with \( (eh_i)^{+} = e \), \( (eh_i)^{\ast} = h_i \), and \( eh_{i}\lll eh_{i+1} \) for \( 1 \leq i \leq n \).   
\end{enumerate}

By Lemma~\ref{eslls}.(3), the second item, \ref{itm:sharp2}, does not hold.

Now, suppose that there are elements \( t_1, \ldots, t_m \) in \( S \), each distinct from both \( f \) and \( ef \), such that  
\( ef \ll t_1 \ll \cdots \ll t_m \ll f.\)
    
Hence, we have
\(
ef = e t_1^{+} \cdots t_m^{+} f t_m^{\ast} \cdots t_1^{\ast} f.
\)
It follows that
\[
ef \ll et_1^{+} \cdots t_m^{+} f t_m^{\ast} \cdots t_1^{\ast} \ll et_1^{+} \cdots t_m^{+} f t_m^{\ast} \cdots t_2^{\ast} \ll \cdots \ll et_1^{+} \ll e.
\]
By above, in this sequence, there is an element equal to \( ef \) somewhere in the middle. Moreover:
\begin{itemize}
    \item every element to the left of it is equal to \( ef \),
    \item every element to the right is equal to \( e \).
\end{itemize}

Since \( ef \neq fe \), we have \( et_1^{+} \cdots t_m^{+} f  = ef \). Then, the occurrence of \( ef \) in the sequence must be this element or one to its right.

Suppose that \( et_1^{+} \cdots t_j^{+} = ef \) and \( et_1^{+} \cdots t_{j-1}^{+} = e \) (if $j=1$, we omit the second condition). Then it must be that:
\[
e\leq t_1^{+}, \ldots, t_{j-1}^{+} = e\quad \text{and} \quad f \leq t_j^{+}, \ldots, t_m^{+}, t_m^{\ast}, \ldots, t_1^{\ast}.
\]
This implies \( t_m = ef \), which contradicts our assumption.

The result follows.
\end{proof}

By Lemma~\ref{splusneqstrastneqLem3}, if $S$ is singleton-rich and $S \not\in \pv{ECom}$, then there exist distinct idempotents $e$ and $f$ such that 
$ef \not\in E(S)$.


\begin{thm}\label{splusneqstrastneqThm}
Let \( S \not\in \pv{ECom} \) be a pseudo $\ll$-transitive semigroup.  
Then there exist elements \( s, t \in S \) such that \( s^{\ast} \neq t^{+} \) and \( s \boldsymbol{*} t \neq 0 \).
\end{thm}

\begin{proof}
By Lemma~\ref{splusneqstrastneqLem3}, there exist distinct idempotents $e$ and $f$ such that $ef \neq fe$, $(ef)^{+} = e$, $(ef)^{\ast} = f$, and there are no elements $s_1,\ldots, s_n$ distinct from $e$, and $ef$ with $ef \ll s_1 \ll \cdots \ll s_n \ll e$ and also there are no elements $t_1,\ldots, t_m$ distinct from $f$, and $ef$ with $ef \ll t_1 \ll \cdots \ll t_m \ll f$.

We prove that \( e \boldsymbol{*} f \neq 0 \).

We have \begin{align*}
\xi(ef,f) 
&=
\sum\limits_{ef\lll s'\lll e}
\big(
\sum\limits_{\substack{f\lll t'\lll f,\\ (ef,f)\ll (s',t')}}
\mu_S(t', f)
\big)
\mu_S(s', e).
\end{align*}

Since \( f \lll t' \lll f \), it follows that \( t' = f \).  
Similarly, from \( ef \lll s' \lll e \), the element \( s' \) must be either \( ef \) or \( e \).  
If \( s' = ef \), then it is straightforward to verify that \( (ef, f) \ll (ef, f) \), since \( \varphi^{(ef, f)} = f \).  
On the other hand, if \( s' = e \), then \( (ef, f) \not\ll (e, f) \), as \( \varphi^{(e, f)} = e \).  
Therefore, we conclude that
\begin{align*}
\xi(ef,f) 
&=
\mu_S(f, f)
\mu_S(ef, e)=-1.
\end{align*}

Let \(e' \in S^{(e)}\) and \(f' \in S^{(f)}\) with $e'f'=ef$ and $\xi(e',f') \neq 0$.

Then, there exist elements $e_1,\ldots,e_n,f_1,\ldots,f_m$ such that 
\[e'\ll e_1 \ll \cdots \ll e_n\ll e\ \text{and}\ f' \ll f_1 \ll \cdots \ll f_m \ll f.\]
Hence, we have
\[e' = {e'}^{+}{e_1}^{+}\cdots {e_n}^{+}e{e_n}^{\ast}\cdots{e_1}^{\ast}{e'}^{\ast}\
\text{and}\
f' = {f'}^{+}{f_1}^{+}\cdots {f_m}^{+}f{f_m}^{\ast}\cdots{f_1}^{\ast}{f'}^{\ast}.\]
It follows that
\begin{align*}
e'f'
&={e'}^{+}{e_1}^{+}\cdots {e_n}^{+}e{e_n}^{\ast}\cdots{e_1}^{\ast}{e'}^{\ast}{f'}^{+}{f_1}^{+}\cdots {f_m}^{+}f{f_m}^{\ast}\cdots{f_1}^{\ast}{f'}^{\ast}\\
&\ll
{e'}^{+}{e_1}^{+}\cdots {e_n}^{+}e{e_n}^{\ast}\cdots{e_1}^{\ast}{e'}^{\ast}{f'}^{+}{f_1}^{+}\cdots {f_m}^{+}f{f_m}^{\ast}\cdots{f_1}^{\ast}\\
&\ll
{e'}^{+}{e_1}^{+}\cdots {e_n}^{+}e{e_n}^{\ast}\cdots{e_1}^{\ast}{e'}^{\ast}{f'}^{+}{f_1}^{+}\cdots {f_m}^{+}f{f_m}^{\ast}\cdots{f_2}^{\ast}\\
&\ll
\cdots\\
&
\ll
{e'}^{+}{e_1}^{+}\cdots {e_n}^{+}e{e_n}^{\ast}\cdots{e_1}^{\ast}{e'}^{\ast} (=e')
\ll
{e'}^{+}{e_1}^{+}\cdots {e_n}^{+}e{e_n}^{\ast}\cdots{e_1}^{\ast}\\
&\ll
{e'}^{+}{e_1}^{+}\cdots {e_n}^{+}e{e_n}^{\ast}\cdots{e_2}^{\ast}
\ll
\cdots
\ll
{e'}^{+}{e_1}^{+}\cdots {e_n}^{+}e
\ll
{e_1}^{+}\cdots {e_n}^{+}e
\ll
\cdots
\ll
e.
\end{align*}

As \( e'f' = ef \), there exists an element equal to \( ef \) in the middle of our sequence.  
Moreover, all elements to its left in the sequence are equal to \( ef \), and all elements to its right are equal to \( e \).

Hence, we have two cases: \( e' = ef \) or \( e' = e \).

First, suppose that \( e' = ef \).  
Then, since
\[
{e'}^{+}{e_1}^{+}\cdots {e_n}^{+}e{e_n}^{\ast}\cdots{e_1}^{\ast}{e'}^{\ast}{f'}^{+} = ef,
\]
we have \( f \leq {f'}^{+} \).  
Similarly, we have \( f \leq {f_i}^{+} \) and \( f \leq {f_i}^{\ast} \) for all \( 1 \leq i \leq m \).  
It follows that \( f = f \). We have already computed that $\xi(e',f')=\xi(ef,f) =-1$

Now, suppose that \( e' = e \).  
Since \( \xi(e', f') \neq 0 \) and \( e' = e \), it follows that \( {f'}^{+} = e \).

The element
\[
{e'}^{+}{e_1}^{+}\cdots {e_n}^{+}e{e_n}^{\ast}\cdots{e_1}^{\ast}{e'}^{\ast}{f'}^{+}{f_1}^{+}\cdots {f_m}^{+}f
\]
is in our sequence.  
This element cannot be equal to \( e \); otherwise, we would have \( e \leq f \), contradicting the assumption that \( (ef)^{\ast} = f \).  
Hence, we must have
\[
{e'}^{+}{e_1}^{+}\cdots {e_n}^{+}e{e_n}^{\ast}\cdots{e_1}^{\ast}{e'}^{\ast}{f'}^{+}{f_1}^{+}\cdots {f_m}^{+}f = ef.
\]

Then, appending \( {f_m}^{\ast} \) to the end gives:
\[
{e'}^{+}{e_1}^{+}\cdots {e_n}^{+}e{e_n}^{\ast}\cdots{e_1}^{\ast}{e'}^{\ast}{f'}^{+}{f_1}^{+}\cdots {f_m}^{+}f{f_m}^{\ast} = ef,
\]
which implies \( f \leq {f_m}^{\ast} \).  
Similarly, we obtain \( f \leq {f_i}^{\ast} \) for all \( 1 \leq i \leq m \).

Thus, there exists an integer \( j \) such that
\[
{e'}^{+}{e_1}^{+}\cdots {e_n}^{+}e{e_n}^{\ast}\cdots{e_1}^{\ast}{e'}^{\ast}{f'}^{+}{f_1}^{+}\cdots {f_j}^{+} = e,
\]
and
\[
{e'}^{+}{e_1}^{+}\cdots {e_n}^{+}e{e_n}^{\ast}\cdots{e_1}^{\ast}{e'}^{\ast}{f'}^{+}{f_1}^{+}\cdots {f_{j+1}}^{+} = ef.
\]

It follows that \( e \leq {f'}^{+}, {f_1}^{+}, \ldots, {f_j}^{+} \), and \( f \leq {f_{j+1}}^{+}, \ldots, {f_m}^{+} \).  
Therefore, we conclude that \( f' = ef \).

We have \begin{align*}
\xi(e',f')=\xi(e,ef) 
&=
\sum\limits_{e\lll s'\lll e}
\big(
\sum\limits_{\substack{ef\lll t'\lll f,\\ (e,ef)\ll (s',t')}}
\mu_S(t', f)
\big)
\mu_S(s', e).
\end{align*}
Since \( e \lll s' \lll e \), it follows that \( s' = e \).  
Similarly, from \( ef \lll t' \lll f \), the element \( t' \) must be either \( ef \) or \( f \).  
If \( t' = ef \), then \( (e, ef) \ll (e, ef) \), since \( \varphi^{(e, ef)} = e \).  
Also, if \( t' = f \), then \( (e, ef) \ll (e, f) \), as \( \varphi^{(e, f)} = e \).  
Therefore, we conclude that
\begin{align*}
\xi(e,ef) 
&=\mu_S(ef, f)\mu_S(e, e)+\mu_S(f, f)
\mu_S(e, e)=-1+1=
0.
\end{align*}

Therefore, the coefficient of \( ef \) in \( e \boldsymbol{*} f \) is non-zero, and the result follows.
\end{proof}


\section{The determinant of \( \lll \)-smooth semigroup}\label{DetNotInEcom}

In this section, we focus on a class of singleton-rich semigroups called \( \lll \)-smooth semigroups and compute their determinants. The definition of \( \lll \)-smooth semigroups extends the concept of \( \ll \)-transitive semigroups given in \cite[Definition 5.4]{Sha-Det2}, providing a precise definition for this class of semigroups.

\begin{deff}\label{ConditionsnEq7}   
Let \(S\) be a pseudo $\ll$-transitive singleton-rich. 
We say that $S$ is $\lll$-smooth if, 
for every sequences                           
\(s''\lll s' \lll s\) and \(t''\lll t' \lll t\),
the following statements hold:
\begin{enumerate}
\item
if \(s''\sharp t''\neq 0\), then we have \(\varphi^{({s''}^{+}s',t'{t''}^{\ast})} = \varphi({s''}^{+}s',t'{t''}^{\ast}).\) 
\item we have \((s'',t'')\ll (s',t')\) if and only if \((s'',t'')\ll (s',t)\).
\item if \(s''\lll s'_1 \lll s'_2 \lll s\) then \((s'',t)\ll (s'_2,t)\) implies \((s'',t)\ll (s'_1,t)\).
\end{enumerate}
\end{deff}

We may extend Lemmas 5.6, 5,8 
in \cite{Sha-Det2} to \( \lll \)-smooth semigroups with an entirely similar proof, leading to the following theorem.
 Hence, we omit the proof in this paper. However, its validity relies on the fact that for elements \( u \) and \( t \) in a \( \lll \)-smooth semigroup \( S \), if \( u \lll t \) with \( u \not\ll t \), then there exist elements \( v_1, \ldots, v_n \) satisfying  
\[
u \ll v_1 \ll \cdots \ll v_n \ll t,
\]  
with \( u^{+} = v_1^{+} = \cdots = v_n^{+} \).  

\begin{thm}\label{TheoremMain}
Let \( S \) be a \( \lll \)-smooth semigroup, and let \( s, t \in S \). The following conditions hold:

\begin{enumerate}
    \item If \( s^{\ast} \neq t^{+} \) and \( s \ast t \neq 0 \), then for every \( t' \in S \) such that \( {t'}^{+} = t^{+} \), the following equivalence holds:
    \[
    s \ast t' \neq 0 \quad \text{if and only if} \quad st^{+} \sharp t' \neq 0.
    \]
    Furthermore, if \( s \ast t' \neq 0 \), then we have
    \[
    s \ast t' = \bigg( \sum_{\substack{st^{+} \lll s' \lll s, \\ t^{+} \leq (s^{+} s')^{\ast}}} \mu_S(s', s) \bigg) st'.
    \]

    \item If \( s^{\ast} = t^{+} \), then the following equivalence holds:
    \[
    s \ast t \neq 0 \quad \text{if and only if} \quad s \sharp t \neq 0.
    \]
    Moreover, if \( s \ast t \neq 0 \), then we have
    \(
    s \ast t = st.
    \)
\end{enumerate}
\end{thm}

Let $X_e=\{x_s\in X\mid s\in \widetilde{L}_e\widetilde{R}_e\}$.
Let $\theta_e(X_e)$ be the determinant of the submatrix $\widetilde{L}_e\times\widetilde{R}_e$ of the Cayley table $(S,\boldsymbol{*})$, for every idempotent $e\in S$. Let $M$ be a matrix that by rearranging and shifting the rows and columns of $C(X)$ so that the elements of the subset $\widetilde{L}_e$ being adjacent rows and the elements of the subset $\widetilde{R}_{e}$ being adjacent columns for every idempotent $e\in E(S)$.
Additionally, we define a matrix $M'$ as follows:
\begin{equation*}
[M'_{r,c}]  = \begin{cases}
  [M_{r,c}]& \text{if}\ r\in \widetilde{L}_e\ \text{and}\ c\in\widetilde{R}_{e}\ \text{for some idempotent}\ e;\\
  0& \text{otherwise},
\end{cases}
\end{equation*} for every row $r$ and column $c$.

By the updated versions of Lemmas 5.6 and 5.8 in \cite{Sha-Det2}, which we reference in Theorem~\ref{TheoremMain}, and using a similar approach as in \cite{Sha-Det2}, it follows that \cite[Lemma 5.10]{Sha-Det2} holds for \( \lll \)-smooth semigroups, which demonstrates the equality of the determinants, i.e., \( \DDet M = \DDet M' \). Based on this, we establish the following theorem for the factorization of the determinant of \( S \).

\begin{thm}\label{Main-TheoremST2t}
Suppose that \(S\) is $\lll$-smooth.
For $s \in S$, put \[y_s=\sum\limits_{t\lll s}\mu_S(t, s)x_t.\] 
Then, we have
\[
\theta_S(X)
= \pm\prod\limits_{e\in E(S)}\widetilde{\theta}_{e}(Y_e)
\]
where $Y_e=\{y_s\mid s \in \widetilde{L}_e\widetilde{R}_e\}$.
Moreover, the determinant of $S$ is non-zero if and only if $\widetilde{\theta}_{e}(Y_e)\neq 0$, for every idempotent $e$.
\end{thm}

In Theorem~\ref{Main-TheoremST2t}, the sign of \(\prod\limits_{e \in E(S)} \widetilde{\theta}_{e}(Y_e)\) depends on whether the number of rearrangements and shifts applied to the rows and columns of \(C(X)\) during the construction of the matrix \(M\) is odd or even.

In the following example, we apply Theorem~\ref{Main-TheoremST2t} to show that the determinant of the semigroup $S$ is non-zero.

\begin{example}\label{Exa2}
Let \( S \) be a semigroup with the following Cayley table.
\begin{center}
\begin{tabular}{ c | c c c c c c c } 
\textbf{\cc{$S$}} & $\cc{y}$ & $\cc{z}$ & $\cc{u}$ & $\cc{t}$ & $\cc{w}$ & $\cc{v}$ & $\cc{q}$\\ \hline 
$\cc{y}$  & . & . & . & . & . & $y$  & $y$ \\ 
$\cc{z}$  & . & . & . & . & $z$  & . & $z$ \\ 
$\cc{u}$  & . & . & . & $y$  & $u$  & $y$  & $u$ \\ 
$\cc{t}$  & . & . & $z$  & . & $z$  & $t$  & $t$ \\ 
$\cc{w}$  & . & $z$  & $z$  & $t$  & $w$  & $t$  & $w$ \\ 
$\cc{v}$  & $y$  & . & $u$  & $y$  & $u$  & $v$  & $v$ \\ 
$\cc{q}$  & $y$  & $z$  & $u$  & $t$  & $w$  & $v$  & $q$ \\ 
\end{tabular}
\end{center}
The semigroup \( S \) is \( \lll \)-smooth, which has been verified using a program developed in C\#.

The table of the semigroup $Z(S)$ 
with the multiplication $\boldsymbol{*}$ is as follows:
\resizebox{\textwidth}{!}{
\begin{tabular}{ l | l l l l l l l } 
\textbf{\cc{$(Z(S),\ast)$}} & $\cc{y}$, & $\cc{z}$, & $y+z+\cc{u}$, & $y+z+\cc{t}$, & $y+\vv{z }+u+t+\cc{w}$, & $\vv{y }+z+u+t+\cc{v}$, & $\vv{y }+\vv{z }+u+t+w+v+\cc{q}$\\ \hline 
$\cc{y}$ & $.$ & $.$ & $.$ & $.$ & $.$ & $y$ & $y$\\ 
$\cc{z}$ & $.$ & $.$ & $.$ & $.$ & $z$ & $.$ & $z$\\ 
$y+z+\cc{u}$ & $.$ & $.$ & $.$ & $y$ & $y+z+u$ & $y$ & $y+z+u$\\ 
$y+z+\cc{t}$ & $.$ & $.$ & $z$ & $.$ & $z$ & $y+z+t$ & $y+z+t$\\ 
$y+\vv{z }+u+t+\cc{w}$ & $.$ & $z$ & $z$ & $y+z+t$ & $y+z+u+t+w$ & $y+z+t$ & $y+z+u+t+w$\\ 
$\vv{y }+z+u+t+\cc{v}$ & $y$ & $.$ & $y+z+u$ & $y$ & $y+z+u$ & $y+z+u+t+v$ & $y+z+u+t+v$\\ 
$\vv{y }+\vv{z }+u+t+w+v+\cc{q}$ & $y$ & $z$ & $y+z+u$ & $y+z+t$ & $y+z+u+t+w$ & $y+z+u+t+v$ & $y+z+u+t+w+v+q$\\ 
\end{tabular}}
We have \( y \lll v,q \) but \( y \not\ll v,q \). Also, \( z \lll w, q \) but \( z \not\ll w, q \).

Starting from the Cayley table of \( (Z(S), \ast) \), we derive the Cayley table of \( (S, \ast) \), which is displayed on the left in the following set of tables. Entries corresponding to elements \( s \) and \( t \) for which \( s^{\ast} \neq t^{+} \) and \( s \ast t \neq 0 \) are highlighted in purple.
\begin{center}
\begin{tabular}{ll} 
\begin{tabular}{ c | c c c c c c c } 
\textbf{\cc{$(S,\ast)$}} & $\cc{y}$ & $\cc{z}$  & $\cc{u}$ & $\cc{t}$ & $\cc{w}$ & $\cc{v}$& $\cc{q}$\\ \hline 
$\cc{y}$ & . & .   & . & . & . & $y$& .\\ 
$\cc{z}$ & . & .  & . & . & $z$  & .& .\\ 
$\cc{u}$ & . & .   & . & $y$  & $u$  & $\qq{-y}$& .\\ 
$\cc{t}$ & .     & . & $z$& . & $\qq{-z}$  & $t$& .\\ 
$\cc{w}$ & .     & $z$  & $\qq{-z}$& $t$  & $w$  & $\qq{-t}$& .\\ 
$\cc{v}$ & $y$      & . & $u$& $\qq{-y}$  & $\qq{-u}$  & $v$& .\\ 
$\cc{q}$ & . & .  & . & . & . & .& $q$ \\ 
\end{tabular} &
\begin{tabular}{ c | c c c c c c c } 
\textbf{\cc{$M$}} & $\cc{y}$ & $\cc{u}$ & $\cc{v}$ & $\cc{z}$ & $\cc{t}$ & $\cc{w}$ & $\cc{q}$\\ \hline 
$\cc{y}$ & . & . & $y$  & . & . & . & .\\ 
$\cc{t}$ & . & $z$  & $t$  & . & . & $\qq{-z}$  & .\\ 
$\cc{v}$ & $y$  & $u$  & $v$  & . & $\qq{-y}$  & $\qq{-u}$  & .\\ 
$\cc{z}$ & . & . & . & . & . & $z$  & .\\ 
$\cc{u}$ & . & . & $\qq{-y}$  & . & $y$  & $u$  & .\\ 
$\cc{w}$ & . & $\qq{-z}$  & $\qq{-t}$  & $z$  & $t$  & $w$  & .\\ 
$\cc{q}$ & . & . & . & . & . & . & $q$ \\ 
\end{tabular}
\end{tabular}
\end{center}
The table $M$ is on the right side that by rearranging and shifting the rows and columns of \( (S, \ast) \) so that the elements of the subset $\widetilde{L}_e$ being adjacent rows and the elements of the subset $\widetilde{R}_{e}$ being adjacent columns for every idempotent $e\in E(S)$. 
The table $M'$ is as follows:
\begin{center}
\begin{tabular}{ c | c c c c c c c } 
\textbf{\cc{$M'$}} & $\cc{y}$ & $\cc{u}$ & $\cc{v}$ & $\cc{z}$ & $\cc{t}$ & $\cc{w}$ & $\cc{q}$\\ \hline 
$\cc{y}$ & . & . & $y$  & . & . & . & .\\ 
$\cc{t}$ & . & $z$  & $t$  & . & . & .  & .\\ 
$\cc{v}$ & $y$  & $u$  & $v$  & . & .  & .  & .\\ 
$\cc{z}$ & . & . & . & . & . & $z$  & .\\ 
$\cc{u}$ & . & . & .  & . & $y$  & $u$  & .\\ 
$\cc{w}$ & . & .  & .  & $z$  & $t$  & $w$  & .\\ 
$\cc{q}$ & . & . & . & . & . & . & $q$ \\ 
\end{tabular}
\end{center}
Then, it is straightforward to compute that the determinant of the matrix \( M' \) is nonzero and equal to \( y^3z^3q \).  
Consequently, the determinant of \( \widetilde{\theta}_{S}(X_{S}) \) is also nonzero and given by  
\[
-y^3z^3(q + t + u - v - w - y - z),
\]  
where \( q \) is substituted by the expression  
\[
q = \sum\limits_{q' \lll q} \mu_S(q', q) q' = (q + t + u - v - w - y - z).
\]  
Moreover, the number of rearrangements and shifts applied to the rows and columns of \( C(X) \) during the construction of the matrix \( M \) is odd.  
Therefore, the sign of  
\(
\prod\limits_{e \in E(S)} \widetilde{\theta}_{e}(Y_e)
\)  
is negative.
\end{example}





\bibliographystyle{plain}
\bibliography{ref-Det}

\end{document}